\documentclass{amsart}

\usepackage{amssymb}
\usepackage{amsmath}
\usepackage{amsthm}

\usepackage{lineno}

\newtheorem {thm}{Theorem}[section]
\newtheorem {pr}[thm]{Proposition}

\newtheorem {lem}[thm]{Lemma}
\theoremstyle{definition}
\newtheorem{df}{Definition}

\theoremstyle{remark}
\newtheorem{rem}{Remark}

\numberwithin{equation}{section}

\newcommand{\al}{\alpha}
\newcommand{\la}{\lambda}

\newcommand{\re}{\mathbb{R}}

\newcommand{\om}{\omega}

\newcommand{\de}{\delta}
\newcommand{\rd}{\mathbb{R}^d}

\begin{document}

\title[On harmonic approximation ]
{  On  harmonic approximation of Lipschitz functions on  compacts in $\mathbb{R}^d$}

\author{Nikolai A.\ Shirokov}

\address{National Research University Higher School of Economics,
 Kantemirovskaya Street 3a, St. Petersburg, 194100, Russia}
\address{St.~Petersburg Department of V.A.~Steklov Mathematical Institute,
Fontanka 27, St.~Petersburg 191023, Russia}
\email{nikolai.shirokov@gmail.com}
\author{Andrei V.\ Vasin}

\address{Admiral Makarov State University of Maritime and Inland Shipping,
Dwinskaya Street 5/7, St.~Petersburg, 198255, Russia}
\email{andrejvasin@gmail.com}


\begin{abstract}
Given a porous compact $K \subset \mathbb{R}^d$ and a continuity modulus $\omega$, we  prove a quantitative Jackson-Bernstein type theorem on harmonic approximation.   That is, a   function $f$ belongs to the class $\mathrm{Lip}_{\omega}(K)$  if and only if $f$ can be approximated uniformly on $K$  with a rate of $\omega(\delta)$  by a function that is  harmonic   in the $\delta$-neighborhood of $K$, provided the uniform estimate $\omega(\delta)/\delta$ on the gradient holds.

\end{abstract}

\maketitle

\section{Introduction}

 \subsection{Background} The original Jackson-Bernstein theorem concerns the relations  between the smoothness of a function and its degree of approximation  by  trigonometric polynomials on the unit circle. Later, the similar problems  of   approximation by analytic polynomials  on the certain compacts in $\mathbb{C}$ were studied (see references in \cite{A4}). In \cite{A2}
  the  characterization by the polynomial approximation of the smoothness of functions  on a Jordan arc $L \subset \mathbb{C} $  involves the uniform
estimates of derivatives of polynomials in a   neighborhood of $L$. The  harmonic polynomials with the uniform estimates on the gradient
can also be used for a constructive description of the Lipschitz classes of functions on the
 Jordan arcs in  $\mathbb{C}$ \cite{A3}. We point out that in all cases the argument uses   conformal mapping techniques.

 Another approach  is developed in  \cite{ASh20} for  the Lipschitz spaces    $\mathrm{Lip}_{\om}(L)$ with  the Dini-regular modulus $\om$ of continuity  on    the  chord-arc curve $L\subset \mathbb{R}^3$. Instead of polynomials  of degree $n$,  the authors used  harmonic functions   in  the $\de$-neighborhood of  $L$  with the  estimates on the gradient.
 A sharp  theorem on relations between   the size $\de$  of a    neighborhood and the rate $\om(\de)$ of approximation is proved.
  In  \cite{P1}  the results were extended    to    the  compacts in $\rd$, which are  bi-Lipschitz images of the $d-2$-dimensional cube.

  The question in order is to extend the results  \cite{ASh20, P1} on harmonic approximation  either on compacts in $\rd$  of dimension greater than $ d- 2$, or for general  moduli of continuity, as well.  We do this work for the porous compacts, which have  Assouad dimension of $\dim _A K < d$ (\cite{L}, see details in Section 2.4). Particularly, the main result holds
   for the  Ahlfors-David $\theta$-regular compacts $K \subset \rd$, $\theta < d$. Also,  our Theorem \ref{t2}  holds for general moduli of continuity without any  Dini type restrictions.

 On the other hand,   Vitushkin type results  in the setting  of harmonic approximation  (see, e.g., \cite{M96, M12, M121})  are closely related to the problem considered here.
We  develop the  approach using the arguments  from   \cite{M12, M121,  P} with those in  \cite{ASh20},
and obtain a quantitative Jackson-Bernstein type  theorem on the relations of the smoothness and degree of approximation by harmonic functions.
\subsection{Main result}
Throughout this paper, we consider $\rd$ equipped with   the
$d$-dimensional Lebesgue  measure.    The Lebesgue  measure of $E$ is denoted by $|E|$, and if $x,\;y\; \in \rd$, then $ |x-y|$
denotes the Euclidean distance  from $x$ to  $y$. The Euclidean distance between arbitrary sets $A,\;B \subset \rd$ is denoted by $\mathrm{dist} (A,B) = \inf\{|x-y|:\; x\in A, \; y\in B\}$.
The open ball with the center $c \in \rd$ and radius $r>0$ is the set
\[B(x, r)= \{y \in \mathbb{R}^d :\; |x-y|<r \}.\]
In this paper, we only consider cubes which are half-open and have sides parallel to the
coordinate axes. That is, a cube in $\rd$  is a set of the form
\[Q = [a_1; b_1)\times \dots   \times [a_d; b_d)\]
with edge size
$\ell(Q) =b_1-a_1 = \dots = b_d-a_d$.
\begin{df}
    A bounded closed set  $K\subset\mathbb{R}^d$ is called  porous, if there  is a  constant $0< c <1$ such that for each ball $B\subset\mathbb{R}^d$
  there exists a ball $mB\subset B$  with the properties $mB\cap K=\emptyset$ and $|mB|\geq c|B|$.
\end{df}
\begin{rem}
  The basic example of a  porous set is   the  Ahlfors-David $\theta$-regular compact $K \subset \rd$, $0<\theta < d$ \cite[Lemma 2.1] {Mat},
  for which there exist positive
constants $C_1,\: C_2$ and $0<\theta<d$ such that for any $x \in K$ and $R>0$
\[
C_1 R^{\theta}\leq \mathcal{H}^{\theta}(K\cap B(x,R))     \leq C_2 R^{\theta},
\]
where $\mathcal{H}^{\theta}$ is  $\theta$-Hausdorff measure and $B(x,R)$ is a ball with the center $x$ and radius $R$.
\end{rem}
\begin{df}
 \label{df1}
  A continuous   increasing  concave function $\omega:[0,\infty)\rightarrow[0,\infty)$ such that $\omega(0)=0$
   is called a  modulus of continuity.  Clearly, the modulus $\om$ satisfies  the doubling property:  for each $\de>0$ it holds $\om(2\de) \leq 2 \om(\de)$.

   The Lipschitz space $\mathrm{Lip}_{\om}(K)$ consists of those functions continuous on $K$ such that the  semi-norm
   \[\|f\|_{\om}= \sup_{\substack { x,y \in K\\
  x\neq y}} \frac{|f(x)-f(y)|}{\om(|x-y|)}\]
   is finite.
               \end{df}
   Given $\de>0$, let
  \[K_{\delta}= \{x:\: \inf_{y \in K}  |x-y| <\de\}\]
   denote a $\de$- neighborhood of $K$, and let $\partial K_{\de}$ be its boundary.
 The main result  is a following Jackson-Bernstein type  theorem on approximation by harmonic functions.
\begin{thm}\label{t2}
 Let $K \subset \rd$, $d>2$, be a porous compact, and let $\omega (t)$ be a continuity modulus. Given a continuous function $f$  on $K$, then   $f \in \mathrm{Lip}_{\omega}(K)$ if and only if
  for any $\delta >0$ there is a harmonic function $\mathcal{G}_{\delta}$ in a $\delta$- neighborhood $K_{\delta}$ of $K$ such that
\begin{equation}\label{ef1}
  \sup_K |f-\mathcal{F}_{\delta}|\leq C_1 \|f\|_{\om}\; \omega(\delta),
\end{equation}
        and
        \begin{equation}\label{ef2}
 \sup_{K_{\de}}|\nabla \mathcal{F}_{\de}|\leq C_2 \|f\|_{\om}\frac{\om (\de)}{\de},
\end{equation}
  where the constants $C_1$  and $C_2$ are independent of $\de$ and $f$.
 \end{thm}
   The theorem generalizes either the results \cite{ASh20} where these problems are researched on   the chord-arc curves in $\mathbb{R}^3$, or  the results \cite{P1} for the  bi-Lipschitz image  in $\rd$ of the $d-2$-dimensional cube.

\begin{rem}
  Observe that  the sufficiency condition of Theorem \ref{t2} does not appeal to the class of harmonic functions.
      If for each  rather small $\de >0 $ one can find any $C^1$-smooth function in $K_{\de}$ with properties (\ref{ef1}) and  (\ref{ef2}) then, obviously, $f \in \mathrm{Lip}_{\om} (K)$. Thus,  we will prove the necessity condition of Theorem \ref{t2}.
\end{rem}

\subsection{Organization and notations}
In Section 2 the preliminaries concerning the Whitney extension, the dyadic decomposition and the properties of porous sets are given. Also, here we reduce Theorem \ref{t2} to the crucial Proposition \ref{pr1}. In  Section 3 we modify the Vitushkin approach and  prove Proposition \ref{pr1}.

As usual, the letter $C$ denotes a constant, which may be different at each
occurrence and which depends on the specified indices under consideration
Also, if  $A/C\leq  B\leq C A,$ then we write $A\approx B$.

\section{Preliminaries}

\subsection{Whitney extension}

From the point of proof, it is convenient to consider  the function $f$ extended from $K$ to $\rd$. This is done by the Whitney extension   of the function
   $f \in \mathrm{Lip}_{\om}(K)$  to  $\rd$ (see Stein \cite[2.2, Ch.6 ]{S}).
   \begin{enumerate}
             \item  $\widetilde{f}=f$ on $K$.
              \item $\widetilde{f}$ has the compact support $supp\; \widetilde{f}\subset R_0$, where $R_0$ is a cube with edge size $\ell(R_0) >1$.
              \item $\widetilde{f} \in C^\infty(\rd \setminus K)\cap \mathrm{Lip}_\om(\rd)$, and there is a  constant  $C >0$ depending on $K$ such that
              \[\|\widetilde{f}\|_{\om}\leq C \|f\|_{\om}.\]

              \item There is a sequence of constants  $C_k >0$,   $k\in \mathbb{N}$, depending on $K$  such that
                  \[|\nabla^k\widetilde{f}(x)|\leq C_k \|f\|_{\om} \frac{\om (\textrm{dist}(x, K))}{\textrm{dist}^k(x, K)},\quad x \in \rd \setminus K.\]
                               \end{enumerate}
                                 Thus, in what follows we assume that $f$ is extended by Whitney to $\rd$ such that   the properties (1)--(4) take place. We  denote  the  Whitney extension $\widetilde{f}$ by the same symbol $f$.
\subsection{Regularization}
 Consider the $C^{\infty}$-regularization of the Whitney extension  $f$ from Section 2.1. For this choose  any function
 $\phi \in C^{\infty}_0(\rd)$ with the support   in the ball $B(0,1)$ centered in origin and radius $1$, and such that $\int_{\rd}\phi(t)dt=1$. Then, a  function
     \begin{equation}\label{ereg}
       f_{\varepsilon}(x)= f\ast \phi_{\varepsilon}(x)=\frac{1}{\varepsilon^d}\int_{\rd}f(x-t)\phi(t/\varepsilon)dt,
     \end{equation}
    provided $\varepsilon < 1/4$,           is the  claimed  $C^{\infty}$-smooth regularization
           with the compact support in the cube
           $ R_0$. For each $\varepsilon >0$ we have
           \[
     \sup_{x \in \rd} |f_{\varepsilon}(x)-f(x)|\leq  \|f\|_{\om}\;\om (\varepsilon),
      \]
             and
             \begin{equation}\label{ereg1}
             \|f_{\varepsilon}\|_{\om} \leq \|f\|_{\om}.
              \end{equation}
              For an integer $k>0$
    we easily  obtain the uniform  estimate on the gradient of the regularization function  $f_{\varepsilon}$
             \[
      |\nabla ^k f_{\varepsilon}(x)| \leq C_k \|f\|_{\om} \frac{\om (\textrm{dist}(x, K))}{\textrm{dist}^k(x, K)}
      \]
with the constant $C_k>0$ independent of  $x \in \rd\backslash K$, $f$ and $\varepsilon > 0$.

Indeed, for  $\varepsilon \leq \textrm{dist}(x ,  K) /2$ we  directly differentiate (\ref{ereg}) with respect to $x$ and apply   the gradient estimate  (4) of the  Whitney extension of $f$ and property (\ref{ereg1}).
If
$\varepsilon > \textrm{dist}(x ,  K) /2$, then choosing any $x_0 \in K$ such that $|x-x_0| = \textrm{dist}(x ,  K)$, we obtain
\[
\aligned
\big |\nabla ^k f_{\varepsilon}(x)\big |
= \big |\nabla ^k f\ast \phi_{\varepsilon}(x)\big |
&= \big |\big(f -f(x_0)\big) \ast \nabla ^k \phi_{\varepsilon}(x)\big |\\
&\leq C_k \|f\|_{\om} \frac{\om (\varepsilon)}{\varepsilon ^k}
\leq C_k \|f\|_{\om} \frac{\om (\textrm{dist}(x, K))}{\textrm{dist}^k(x, K)},
\endaligned
\]
where in the last line  the concavity of $\om$ is applied. This completes the gradient estimate of $f_{\varepsilon}$.
\subsection{Reduction of proof}
By the regularization  from Section 2.3, we will prove the necessity part assuming $f \in C^{\infty}_0(\rd)$. The next strengthened form of estimate (\ref{ef1}) holds.
   \begin{pr}\label{pr1}
     Given a  cube $R_0$, a porous compact $K  \subset R_0$,     a continuity modulus $\omega (t)$ and  a function   $f \in C^{\infty}_0(\rd)$ with   $supp\; f \subset R_0$.   Then, for each $0<\de<  1/4$  there exists a harmonic function $\mathcal{F}_{\delta}$  in  the   neighborhood $K_{\de}$ such that
     \begin{equation}\label{ef11}
  \sup_{K_{\de}} |f-\mathcal{F}_{\delta}|\leq C \|f\|_{\om}\; \omega(\delta),
\end{equation}
  where the constant $C>0$ is independent of $f$ and $\de$.
   \end{pr}
  Before proving  Proposition \ref{pr1}, observe that the necessity part of Theorem \ref{t2} easily follows.
   Indeed, by the triangle inequality and by the regularization from Section 2.3, estimate (\ref{ef1}) is the direct corollary of (\ref{ef11}).

  The gradient estimate  (\ref{ef2})  modulo regularization is proved as follows.  After redesignation, we may assume that for each $\de>0$ there is the harmonic function $\mathcal{F}_{\de}$ in the $2\de$- neighborhood $ K_{2\de}$
of $K$ such that (\ref{ef11}) holds.
By the triangle inequality,  there is a constant $C$ such that for each $\de>0$  one has
\begin{equation}\label{est}
  \sup _{\substack { x,y \in K_{2\de}\\
|x-y|<\de }} | \mathcal{F}_{\de}(x)-\mathcal{F}_{\de}(y)| \leq C \|f\|_{\om}\;\om(\de).
\end{equation}
Differentiating the  Poisson formula for the harmonic function $\mathcal{F}_{\de}$ in the ball $B=B(x, \de) \subset K_{2\de}$,
 one clearly obtains
\[
\aligned
\nabla_t \mathcal{F}_{\de}(t)\bigg|_{t=x}
&=\nabla_t (\mathcal{F}_{\de}(t)-\mathcal{F}_{\de}(x))\bigg|_{t=x}\\
&= \frac{C_d}{\de}\int_{|y-x|=\de}\nabla_t\frac{\de^2-|x-t|^2}{|y-t|^d}\bigg|_{t=x}(\mathcal{F}_{\de}(y)-\mathcal{F}_{\de}(x)) dS(y),\\
\endaligned
\]
 where $dS(y)$ is the induced surface measure on $\partial B$.
Therefore, estimating the derivatives of the  Poisson kernel and using (\ref{est}), it holds
\[
\aligned
|\nabla \mathcal{F}_{\de} (x)|
\leq C\|f\|_{\om}\;\int_{|y-x|=\de} &\frac{\om(\de)}{\de^{d}} dS(y)\\
&\leq C\|f\|_{\om} \frac{\om(\de)}{\de},
\endaligned
\]
with a constant  $C$  independent of  $\de>0$ and $x \in  K_{\de}$. Thus, (\ref{ef2}) follows.
\begin{rem}
Observe, that   we  may prove  Proposition \ref{pr1}   for a discrete positive sequence  of reals  $\de_j \rightarrow 0$ with the constant  $C$ independent of $\de_j$. This easily gives  (\ref{ef11})   for all $\de >0$.
\end{rem}

\subsection{Porous sets}

Given a set $E \subset \rd$, then the Assouad dimension
$\dim_A (E)$ of $E$ (see, for instance, \cite{L}) is the infimum of all exponents $\lambda \geq 0$ for which there exists a constant $ C_{\lambda}$ such that for every ball $B(x,\; R)$ with the center $x \in E$ and radius $R$, and  for each $0 < r < R$  the set $E\cap B(x, R)$ can be covered by at most $\left(\frac{R}{r}\right)^{\lambda}$ balls of radius $r$.
 Also,  a set $E \subset \rd$ is porous if and only if
$\dim_ A(E) < d$.

 Instead of balls, also dyadic cubes could be used in the construction leading to the Assouad dimension. The dyadic formulation  is convenient from the point
of view of our proof.
The dyadic decomposition of each cube  $Q_0$ is
\[ \mathcal{D}(Q_0)= \bigcup_{j\geq 0} \mathcal{D}_j(Q_0),\]
where each $\mathcal{D}_j(Q_0)$ consists of the $2^{jd}$  pairwise disjoint (half-open) cubes $Q$, with side length
$\ell(Q) = 2^{-j}\ell(Q_0)$, such that
\[Q_0= \bigcup_{Q \in \mathcal{D}_j(Q_0)} Q\]
for every $j =0,1, \dots $. The cubes in  $\mathcal{D}(Q_0)$ are called dyadic cubes (with respect to $Q_0$).
 Let  $\mathrm{Card} \; ( \mathcal{A})$ denote  the number of all cubes  $Q$  from  a family $\mathcal{A}$.

For a compact  $K\subset\rd$, a cube  $R\subset\rd$  and   a  positive integer  $j$ define the covering of $R\cap K$  by pairwise disjoint  cubes
\[ {D}_{j, K}(R)=\big\{Q \in {D}_{j}(R):  \; Q  \cap K  \neq \emptyset\big\}.\]
Arguing as in \cite[Theorem 5.2]{L} where  Assouad dimension is defined,  one can prove the lemma.
\begin{lem}\label{l_cov}
  Given a porous set $K$, there exist two constants $0\leq \la < d$ and $C_{\la}>0$ such that for each cube $R \subset \rd$, for a positive integer $ j$, and the family $\mathcal{D}_{j, K}(R) $,
it holds
\[
\mathrm{Card}\;(\mathcal{D}_{j, K}(R))
 \leq C_{\la} 2^{j \lambda}.
\]
\end{lem}
\begin{proof}

 Let $c$ be  the porosity constant, and let $k_0$ be a  positive integer, such that
 $ 2 ^{-k_0} \leq c< 2 ^{-k_0 +1}$. Define $\theta = 2 ^{-k_0}$.   There is  a dyadic cube $Q \in \mathcal{D}_{k_0} (R)$ such that $Q \cap K =\varnothing$. Choose the constant  $0\leq \la< d$ such
$2^{k_0 d}-1= 2^{k_0 \la}$.
  Then $K \cap R$ is covered by at most
  \[
 \left  (\frac{\ell(R)}{\ell(Q)}\right )^d -1= 2^{k_0 d}-1 =  2^{k_0 \la}
  \]
 cubes from $\mathcal{D}_{k_0} (R) $.
By induction  for each integer $n>0$ we cover $K \cap R$ by
 at most  $2^{n k_0 \la}$ cubes from $\mathcal{D}_{nk_0} (R) $.

 Now, for an integer $j>0$  let $K \cap R$ be covered  by a family  of pairwise disjoint cubes from $\mathcal{D}_{j}(R)$.  Define the integer $n$ such that
  \[k_0\; n\leq j< k_0 (n+1). \]
  Clearly, for each cube $Q \in \mathcal{D}_{j, K}(R)$ there exists a cube
   $Q' \in
  \mathcal{D}_{(n+2)k_0, K} (R)$ such that
  $Q' \subset Q $.
     Therefore, the  number $\mathrm{Card}\;( \mathcal{D}_{j, K}(R))$ of  cubes from the family $ \mathcal{D}_{j, K}(R)$ is estimated above as follows
 \[
\aligned
 \mathrm{Card}\;( \mathcal{D}_{j, K}(R))
 & \leq \mathrm{Card} \;(\mathcal{D}_{(n+2)k_0, K} (R))\\
 &\leq 2^{(n+2) k_0 \la}\leq   2^{2k_0 \la}\; 2^{j \la},
 \endaligned
\]
that is required with the constant $C_{\la}=2^{2k_0\la} $.
\end{proof}

\section{Proof of Proposition \ref{pr1}}
\subsection{Partition of unity}  We construct the next partition of unity in $\rd$ (see for instance \cite[\S 7, Ch.VIII]{G}, \cite {V}).
Consider the disjoint family $\mathcal{D}_0$ of half open cubes $Q$ with vertexes in $\mathbb{Z}^d$ and edge size
$\ell(Q)=1$. Namely, these are  all cubes of the form
\[Q = [a_1; a_1+1)\times \dots   \times [a_d; a_d+1),\;\;a_1, \dots , a_d \in \mathbb{Z}.\]
Let
\[
\mathcal{D}_j= \bigcup_{Q \in \mathcal{D}_0} \mathcal{D}_j(Q).
\]
For an integer $j\geq 0$  consider a partition of unity   assigned to the dyadic  family $\mathcal{D}_j$ from Section 2.4. Namely, this   is the  family
$\Phi_j =\{\phi_Q, \;Q \in \mathcal{D}_j\}$ of bump functions  with the properties:
\begin{enumerate}
  \item $\phi_Q \in C^{\infty}_0(\rd)$;
  \item $supp \;\phi_Q \subset 2Q$, where $2Q$ is the dilated cube to $Q$ with the same center, and $\ell(2Q)=2\ell(Q)$;
  \item  \[\sum _{Q\in \mathcal{D}_j} \phi_Q  = \chi_{\rd};\]
  \item  there is a sequence of constants  $C_k>0$, $ k =0, 1, \dots$,  independent of   $Q$  and $\phi_Q $ such that
  \[
  \|\nabla^k\phi_Q\|_\infty \leq \frac{C_k}{\ell^k(Q)}.
  \]
\end{enumerate}
To construct $\Phi_j$ consider  a bump function
$\phi \in C^{\infty}_0(\rd)$ with $\int_{\rd} \phi (x) dx= 1$ and $supp\;\phi \subset B(0, 1/2)$. Define the sequence of constants as
\[C_k = \|\nabla^k\phi\|_\infty,\; k =0, 1, \dots.\]
 Then
the family
$\Phi_0= \{\phi_Q, \; Q \in \mathcal{D}_0\}$  of functions $\phi_Q (x)  =\int_{Q} \phi (x-y) dy$, $Q \in \mathcal{D}_0$, will be the claimed partition of unity assigned to $\mathcal{D}_0$ for $j=0$. Taking an integer $ j>0$ and  a cube $Q \in \mathcal{D}_j$ there is  the    unique cube $\widetilde{Q} \in \mathcal{D}_0$ dilated to $Q$  with respect to origin. Putting
\[\phi_{Q} (x)  = \phi_{\widetilde{Q} } (2^jx),\]
we clearly obtain, that  for the family $\Phi_j =\{\phi_Q\}$ all properties (1)--(4) hold.
\subsection{ Vitushkin type localization}
Let $f \in C^{\infty}_0(\rd)$ and let $\mathcal{E}(t)= C_d  |t|^{-d+2}$ be a fundamental solution of the Laplasian $\Delta$. Fix $j=0,1, \dots  $, then for the  partition of unity $\Phi_j=\{\phi_Q: \; Q \in \mathcal{D}_j\}$
 from  Section 3.1,  we have
  \[
\aligned
f(x)=\mathcal{E}\ast \Delta f
= \sum_{Q \in \mathcal{D}_j} \mathcal{E}\ast &(\phi_Q \Delta f )(x)
= \sum_{Q\in \mathcal{D}_j} V_{\phi_Q} f(x).
\endaligned
\]
 The  operator defined as
\[
\aligned
V_{\phi_Q} f (t)
=\quad \mathcal{E}\ast  (\phi_Q \Delta f )& (t)
 = C_d \int_{\rd} \frac{   \phi_Q(y) \Delta f(y) }{|t-y|^{d-2}} dy
\endaligned
\]
 with the constant $C_d$, is called the localization  assigned to $\phi_Q$ \cite{M96, M121, P}.

  Observe, that  since $f \in C^{\infty}_0(\rd)$,
  the localization $ V_{\phi_Q} f$ is well defined in the classical sense. Further, $ V_{\phi_Q} f$ is harmonic outside the support of $ \phi_Q $. Also, $ V_{\phi_Q} f=0$ if $\mathrm{supp}\, \phi_Q \cap \mathrm{supp} \,f= \varnothing$, and therefore,
  $V_{\phi_Q} f=0$ if $\mathrm{supp} \,\phi_Q \cap R_0 = \varnothing$.
  In what follows we will approximate each function $ V_{\phi_Q} f$ by a harmonic function in a   neighborhood of $K$. We have a modification of Lemma 5 \cite{P}.
\begin{lem}\label{lloc1}
  Let $f \in C^\infty_0(\rd)$ and  $\om$ be a modulus of continuity.   There is a constant $C$ depending on $d$  such that for each bump function $\phi_Q \in \Phi_j$, $Q \in \mathcal{D}_j, \;j=0,1, \dots $, and $\ell=\ell(Q)= 2^{-j}$, it holds
      \begin{equation}\label{eloc}
  \sup_{\rd}|V_{\phi_Q}f|\leq C \|f\|_{\om}\om ( \ell).
\end{equation}
\end{lem}

\begin{proof}
The function $V_{\phi_Q}f$ is harmonic in   $\rd \backslash  2Q$ and vanishes at $\infty$, therefore by the maximum principle, it is sufficient to prove the estimate (\ref{eloc}) in $2Q$.

 By the elementary differentiating formula
\[fg''=(fg)''+ f''g-2(f'g)',\]
one has for $t \in 2Q$
  \[
\aligned
V_{\phi_Q}(f)(t)
&= \mathcal{E}\ast \phi_Q \Delta (f-f(t)) (t)\\
  & = \mathcal{E}\ast \Delta \big(\phi_Q (f-f(t))\big)(t)+ \mathcal{E}\ast (f-f(t)) \Delta \phi_Q (t)\\
  & -2 \sum_{i=1}^d \mathcal{E}\ast  \partial_{y_i}\big((f-f(t))\partial_{y_i}\phi_Q \big) (t)\\
   & = \;I_1\;+\; I_2\;+\;I_3.
   \endaligned
\]
For the first term $I_1$,  we easily have
\[
\aligned
I_1
=\mathcal{E}\ast \Delta \big(\phi_Q (f-f(t))\big)(t)
 = \phi_Q \big(f-f(t)\big)(t)=0.
 \endaligned
 \]
Secondly, term $I_2$ is an integral with the weak singularity and by property (4) from Section 3.1 of the bump function and by  the $\mathrm{Lip}_{\om}$-estimates of $f$, one has
  \[
  \aligned
|I_2|
& \leq C_d \int_{2Q} \frac{ | \Delta \phi_Q(y)| |f(y)-f(t)| }{|t-y|^{d-2}} dy\\
& \leq C_d  \|f\|_{\om} \frac{\om(\ell)  } { \ell^2 }  \int_{2Q} \frac{ dy }{|t-y|^{d-2}}
 \leq C_d  \|f\|_{\om}\;\om(\ell)
\endaligned
\]
uniformly for $t \in 2Q$ with a constant $C_d$ depending on $d$.

Finally, term $I_3$ is the sum of  integrals with the weak singularity
 \[
\aligned
I_{3,i}
&= C_d\;\mathcal{E}\ast \partial_{y_i} \bigg(\partial_{y_i}\phi_Q \big(f-f(t)\big)\bigg) (t)\\
&=  C_d \int_{\rd} \frac{\partial_{y_i}\big(\partial_{y_i}\phi_Q (y) (f(y)-f(t))\big)}{|t-y|^{d-2}}dy.
\endaligned
\]
We apply the Green formula to $I_{3,i}$ for each $i=1, \dots, d$.  So, integrating with respect to  the domain $ 2Q \cap \{ y: |y-t|>\varepsilon\}$ and  taking into account that $supp\;\phi_Q \subset 2Q$, we obtain
\[
\aligned
I_{3,i, \varepsilon}
& =C_d\int_{ \substack {y \in 2Q \\ |y-t|>\varepsilon}}  \frac{\partial_{y_i}\big(\partial_{y_i}\phi_Q(y) (f(y)-f(t))\big)}{|t-y|^{d-2}}dy\\
 &= -C_d\int_{ \substack {y \in 2Q \\ |y-t|>\varepsilon}} \partial_{y_i}\frac{1}{|t-y|^{d-2}} \partial_{y_i}\phi_Q(y) \big(f(y)-f(t)\big)dy \\
 & + C_d \int_{ |y-t|=\varepsilon} \frac{\partial_{y_i}\phi_Q(y) \big(f(y)-f(t)\big)}{|t-y|^{d-2}}  \cos(\nu, y_i) d S(y)\\
 & =I'_{3,i, \varepsilon} \;+ \;I''_{3,i, \varepsilon}
\endaligned
\]
where $dS(y)$ is the induced $d-1$-dimensional surface  measure,   and where  $\nu=\nu(y)$ denotes the outer normal  to  the  sphere $\{ y: |y-t|=\varepsilon\}$.   We estimate both the integrals, as follows
\[
\aligned
|I'_{3,i, \varepsilon}|\;
&\leq\; C_d \|f\|_{\om} \frac{\om(\ell)}{\ell}  \int_{  \substack {y \in 2Q \\ |y-t|>\varepsilon}}  \frac{1}{|t-y|^{d-1}} dy
\leq \;C_d  \|f\|_{\om} \;\om (\ell)
\endaligned
\]
and respectively,
\[ \aligned
|I''_{3,i, \varepsilon}|
&\leq C_d  \|f\|_{\om}\frac{\om(\varepsilon)}{\ell}\int_{ |y-t|=\varepsilon }   \frac{d S(y)}{\varepsilon^{d-2}}
  \leq C_d  \|f\|_{\om} \frac{\varepsilon\om(\varepsilon)}{\ell}
  \leq C_d  \|f\|_{\om} \om (\ell).
\endaligned
\]
Summing
\[| I_{3}|\leq C_d \sum_i | I_{3,i}| \leq  C_d  \|f\|_{\om}\;\om(\ell),\]
we complete the proof of (\ref{eloc}) with a constant $C_d$ depending on $d$.
\end{proof}

 \subsection{Taylor expansion of localization}

  Let
\[\mathcal{D}_{j,K}=\{Q \in \mathcal{D}_j: \:  Q \cap K \neq \emptyset\}\]
be a family of all dyadic cubes of edge size $\ell = 2^{-j}$ which cover $K$.
Let $\mathcal{D}'_{j,K} \supset \mathcal{D}_{j,K}$ be the family of those cubes $Q' \in  \mathcal{D}_j $ for which  there exists  a  neighbor cube $Q \in \mathcal{D}_{j,K}$. We call two dyadic cubes as  neighbors if their closures have non-empty intersection. Clearly, the set
$\bigcup\{Q';\; Q' \in \mathcal{D}'_{j,K}\}$
contains  a $2^{-j}$-neighborhood of
 $\bigcup\{Q;\; Q \in \mathcal{D}_{j,K}\}$.

     By the porosity condition, for each $Q \in \mathcal{D}_{j,K}$ one can  take a ball $B_Q \subset Q$ with the center $c_{B_Q}$, radius $r>0$ such that
   $r\geq c \ell$,  and  Euclidean distance   $ dist(B_Q, K) \geq 2r$,
   where $0<c< 1$ is a porosity constant.  The same is trivial for all cubes $Q \in \mathcal{D}_{j} \backslash \mathcal{D}_{j,K}$, since these cubes do not intersect $K$. Particularly, it holds for all $Q \in \mathcal{D}'_{j,K}$.

    For each ball $B_Q$,  $Q\in \mathcal{D}'_{j,K}$, consider the Taylor expansion of the fundamental solution  $\mathcal{E}(t-y)$ at $\infty$ with a pole in the center  $c_{B_Q}$ of ${B_Q}$. Namely, define a multi-index
     $\al \in \mathbb{Z}^d$, $\al =(\al_1,\dots, \al_d)$, $|\al|= \sum_k \al_k$, and $\al!=\al_1!\dots \al_d!$.
    We need  a modified  estimate from \cite{P}.
      \begin{lem}\label{ldif}
     There is a constant $C_d$ depending on dimension $d$ such that for $t =(t_1,\dots, t_d)\in \rd\backslash \{0\}$
     with the Euclidean norm $|t|$, it holds
   \begin{equation}\label{efe}
 \left |  \partial^{\al} \frac{1}{|t|^{d-2}}\right |  \leq \;C_d \; \al!\;|\al|^{\frac{d-1}{2}} \left ( \frac{2d}{|t|}\right)^{|\al| +d -2},
\end{equation}
   \end{lem}
  \begin{proof}
    Let $z \in \mathbb{C}^d$,
     $z=(z_1,\dots, z_d)$, where  $z_k=x_k + i y_k$, $k=1,\dots, d$, and  let  $|z|= (|z_1|^2+\dots + |z_d|^2)^{1/2}$.

    For each $t  \in \rd\setminus 0$,   the fundamental solution   $\mathcal{E}(t-x)=C_d|t-x|^{2-d}$ considered as the function  of   $x \in \rd\setminus t $, may be extended as holomorphic function
   $\mathcal{E}(t-z)$ for $z \in \mathbb{C}^d$ such that  $|z|< |t|/2$.

  By the  Cauchy-Leray integral formula \cite[\S 3.2]{Ru}, it holds
  \[ \mathcal{E}(t-z)= C_d \frac{|t|}{2}\int_{|w|=|t|/2} \frac{\mathcal{E}(t-w)}{(|t|^2/4 - \langle w, z\rangle)^{d}}dS(w), \]
 where for  $w=(w_1,\dots, w_d) \in \mathbb{C}^d$
 \[\langle w, z\rangle= \sum_{k=1}^{d} w_k \overline{z_k},\]
  and $dS$ is the induced surface measure on the $2d-1$-sphere $|w|= |t|/2$.
  Differentiating, one has
  \[
\aligned
\frac{\partial^{\al}\mathcal{E}(t)}{\partial t^{\al}}
&=(-1)^{|\al|}\frac{\partial^{\al}\mathcal{E}(t-x)}{\partial x^{\al}}\bigg|_{x=0}\\
&= (-1)^{|\al|}\frac{\partial^{\al}\mathcal{E}(t-z)}{\partial z^{\al}}\bigg|_{z=0}\\
&=     C_d(-1)^{|\al|} \frac{(d+|\al|-1)!}{(d-1)!} \frac{|t|}{2}\int_{|w|=|t|/2} \frac{\mathcal{E}(t-w)\;\overline{w^{\al}}}{(|t|^2/4 - \langle w, z\rangle)^{d+|\al|}}d S(w)\bigg|_{z=0}\\
&=     C_d (-1)^{|\al|}\frac{(d+|\al|-1)!}{(d-1)! } \left (\frac{|t|}{2}\right)^{1-2|\al|-2d}\int_{|w|=|t|/2} \mathcal{E}(t-w)\;\overline{w^{\al}}\;dS(w).
\endaligned
\]
Therefore, estimating  the above integral
\[\int_{|w|=|t|/2} \mathcal{E}(t-w)\;\overline{w^{\al}}\;dS(w) \leq C_d \left (\frac{2}{|t|}\right)^{|\al|+d+1},\]
we have
\begin{equation}\label{efac}
  \bigg|\partial^{\al}\frac{1}{|t|^{d-2}}\bigg |
\leq
C_d \frac{(d+|\al|-1)!}{(d-1)!} \left (\frac{2}{|t|}\right)^{d+|\al|-2}.
\end{equation}
Now, by Stirling's approximation of  the Euler $\Gamma $-function, it holds
 \[ (d+|\al|-1)! \leq C_d |\al|! \;|\al|^{d-1}.\]
and
\[
\aligned
|\al|!
& =\Gamma (|\al|+1)  \leq C_d \left (\frac{|\al|}{e}\right )^{|\al|}\sqrt{2\pi |\al|}.
\endaligned
\]
Combining the same argument with
the  logarithmic convexity of the $\Gamma $-function  implies
\[
\aligned
\al!
 =\Gamma (\al_1+1)\dots \Gamma (\al_d+1)
 & \geq \Gamma^d \big(|\al|/\;d +1\big) \\
&  \geq C_d \left (\frac{|\al|}{e\;d}\right )^{\frac{|\al|}{d}d} \big(\sqrt{2\pi |\al|/d}\big)^d.
\endaligned
\]
Therefore, we have
\[
\aligned
|\al|!
  \leq C_d \;\al!\;|\al|^{\frac{1-d}{2}}\;d^{ |\al|}.
\endaligned
\]
Substituting the result in    (\ref{efac}), we complete the proof of the lemma.
  \end{proof}
 By  Lemma \ref{ldif},  the Taylor series
  \[
  \frac{1}{|t-y|^{d-2}}= \sum _{|\al|\geq 0 } \frac{(-1)^{|\al|}}{\al!} (y-c_{B_Q})^{\al}  \partial^{\al} \frac{1}{|t-c_{B_Q}|^{d-2}}
  \]
 converges  uniformly if  $4 d|y-c_{B_Q}|< |t-c_{B_Q}|$.
Hence whenever we take $y \in 2Q$ and obtain 
 \[|y-c_{B_Q}| \leq  2 \ell \sqrt{d},\]
 then for 
 \[|t-c_{B_Q}|> 8 \ell d \sqrt{d} \]
 we  can  represent uniformly each localization $V_{\phi_Q}f$,  $\;Q \in \mathcal{D}_{j}$
  as follows
\begin{equation}\label{etay}
 V_{\phi_Q}f (t)= \sum _{|\al|\geq 0 }C_{\al, Q} \partial^{\al} \frac{1}{|t-c_{B_Q}|^{d-2}} ,
 \end{equation}
where
\[
\aligned
C_{\al,Q}
&= C_d \frac{(-1)^{|\al|}}{\al!} \int_{2Q}\phi_Q(y) (y-c_{B_Q})^{\al} \Delta f(y) dy\\
& = C_d \frac{(-1)^{|\al|}}{\al!} \int_{2Q}\Delta \big((\phi_Q(y) (y-c_{B_Q})^{\al} \big)\big(f(y)-f(c_{B_Q})\big)dy.
\endaligned
\]
Estimating   the derivatives  of the bump function $\phi_Q$ by property (4) from Section 3.1, we obtain
\begin{equation}\label{ebq0}
  |C_{\al,Q}| \leq C_d \frac{1}{\al!} \|f\|_{\om}\;\om(\ell)(2\ell \sqrt{d})^{d+|\al|-2},
\end{equation}
where a constant $C_d$ depends on $d$ and is independent of parameters  $\al$, $\ell$ and $\|f\|_{\om}$.
 \subsection{Approximation function construction}
For each $Q \in \mathcal{D}'_{j,K}$ we define a harmonic function in $\rd\backslash c_{B_Q} $ as the next finite sum
\[
F_{Q}(t)=\frac{C_{0,Q}}{|t-c_{B_Q}|^{d-2}}+  \sum _{|\al|=1}C_{\al,Q}  \partial^{\al} \frac{1}{|t-c_{B_Q}|^{d-2}}.
\]
Clearly, $F_{Q}$ has a pole in the center $c_{B_Q}$ of the ball $B_Q \subset Q$. By (\ref{ebq0}) and (\ref{efe}),
\[
\aligned
 \sup_{ K} |F_{Q}(t)|
  &\leq C\sup_{K} \|f\|_{\om}\;\om(\ell)\left(\frac{\ell^{d-2}}{|t-c_{B_Q}|^{d-2}}+\frac{\ell^{d-1}}{|t-c_{B_Q}|^{d-1}}\right)\\
  &\leq C \|f\|_{\om}\;\om(\ell)
\endaligned
\]
with the constant  $C$ independent of $Q$ and $ \|f\|_{\om}$.

For each scale parameter $\ell=2^{-j}$ define the  function
\[F_{\ell}=\sum_{Q\in  \mathcal{D}'_{j,K}} F_{Q}, \]
which is harmonic  in $\rd$  except the  finite number of poles in centers $ c_{B_Q}$ of  balls $B_Q$, $B_Q \subset Q  \in \mathcal{D}'_{j,K}$. Recall,  that all  poles $ c_{B_Q}$ are separated from $K$ by at least  than $ 2c\ell$ uniformly with respect to  $Q \in  \mathcal{D}'_{j,K}$, provided $c$ is a porosity constant. Particularly, it holds  that $F_{\ell} $ is harmonic in the  $2c\ell$- neighborhood of $K$.
\subsection{Proof of Proposition   \ref{pr1}}
Let us consider  $f \in C^{\infty}_0(\rd)$  and expand it as a finite sum of the  Vituschkin localizations from Section 3.2
\[
\aligned
f(x)
&=\sum_{Q\in \mathcal{D}_j} V_{\phi_Q} f(x)\\
&=\sum_{Q\in \mathcal{D}'_{j,K}} V_{\phi_Q} f(x) + \sum_{Q\in  \mathcal{D}_j\backslash \mathcal{D}'_{j,K}} V_{\phi_Q} f(x)\\
&= \quad I_{\ell}(x) \quad + \quad O_{\ell}(x).
\endaligned
\]
Clearly,  if $Q_1 \in \mathcal{D}_{j,K}$ and $Q_2 \notin \mathcal{D}'_{j,K}$, then interiors of $2Q_1$ and $2Q_2$  have empty intersection. Therefore, $\mathcal{O}_{\ell}(x)$
 is harmonic in the interior of $\bigcup \{2Q:\;Q \in \mathcal{D}_{j,K}\}$, which contains the $\ell$-  neighborhood
  $ K_{\ell}$ of the compact $K$.

Therefore, it remains to approximate $I_{\ell}$. For this we apply  the function  $F_{\ell}$, constructed in Section 3.4,  and estimate the difference
 \[
\aligned
I_{\ell}(t)-F_{\ell}(t)
&=  \sum_{Q\in \mathcal{D}'_{j,K}} \big(V_{\phi_Q} f(t) - F_{Q}(t)\big)  \\
\endaligned
\]
for $t \in K_{c\ell}$, provided $ K_{c\ell}$ is the   $c\ell$- neighborhood of $K$, and $c$ is the porosity constant.

 Put $t \in K_{c\ell}$, and split the sum according to whether for each ball $B_Q$  with the center $c_{B_Q} $ the distance $|t-c_{B_Q}|$  is less or greater than $ 8\ell d \sqrt{d}  $:
\[
\aligned
I_{\ell}(t)
&-F_{\ell}(t)\\
&=  \sum_{\substack { Q \in \mathcal{D}'_{j,K}\\
 |t-c_{B_Q}|\leq 8 \ell d \sqrt{d} }} \big(V_{\phi_Q} f(t) - F_{Q}(t)\big)
+  \sum_{\substack {  Q \in\mathcal{D}'_{j,K}\\
 |t-c_{B_Q}|>8\ell d \sqrt{d} }}\big( V_{\phi_Q} f(t) - F_{Q}(t)\big)\\
&= \qquad I_1 \qquad+\qquad I_2.\\
\endaligned
\]
        We start with term $I_1$. By the size property (\ref{eloc}) of $V_{\phi_Q} f $ and the similar  property in Section 3.4 of $F_Q$, we have

\[
\aligned
 |I_1|
 &\leq\sum_{\substack {  Q \in \mathcal{D}'_{j,K}\\
  |t-c_{B_Q}|\leq 8 \ell d \sqrt{d} }} |V_{\phi_Q} f (t)- F_{Q}(t)|\\
 & \leq   \sum_{\substack {  Q \in \mathcal{D}'_{j,K}\\
  |t-c_{B_Q}|\leq 8\ell d \sqrt{d} }} |V_{\phi_Q} f(t)|+ \sum_ {\substack {  Q \in \mathcal{D}'_{j,K}\\
  |t-c_{B_Q}|\leq 8 \ell d \sqrt{d} }} |F_{Q} (t)|\\
 & \leq C  \sum_{\substack {  Q \in \mathcal{D}'_{j,K}\\
  |t-c_{B_Q}|\leq 8 \ell d \sqrt{d} }}\|f\|_{\om}\; \om(\ell)\\
 &  \leq C\|f\|_{\om}\; \om(\ell) \sum_{\substack {  Q \in \mathcal{D}'_{j,K}\\
  |t-c_{B_Q}|\leq 8 \ell d \sqrt{d} }} 1,
    \endaligned
\]
where $C$ depends on $K$.
The last sum equals to  the number of cubes  of the family
\[ \{Q \in \mathcal{D}'_{j,K}:\;|t-c_{B_Q}|\leq 8 \ell d\sqrt{d} \},\]
which is easily estimated above  by the   constant $C_d$ depending on dimension  $d$ uniformly with respect to  $t \in K_{c\ell}$ and $\ell>0$. Thus,
  \[
  |I_1|\leq C_d \|f\|_{\om}\;\om(\ell),
  \]
 which is required for term $I_1$.

The estimate of the second sum is more complicated. For $V_{\phi_Q} f$ we use the Taylor formula (\ref{etay}) when $ |t-c_B|>8 \ell d\sqrt{d} $,  and estimate the sum of remainders
\[
\aligned
I_2
&=\sum_{\substack {  Q \in \mathcal{D}'_{j,K}\\
 |t-c_{B_Q}|>8 \ell d \sqrt{d} }} \big(V_{\phi_Q} f (t)- F_{Q}(t)\big)\\
&=\sum_{\substack  {  Q \in \mathcal{D}'_{j,K}\\
 |t-c_{B_Q}|>8 \ell d \sqrt{d} }} \left( \sum _{|\al|\geq 0 }C_{\al, Q} \partial^{\al} \frac{1}{|t-c_{B_Q}|^{d-2}}  -
\sum _{|\al|\leq1}C_{\al,Q}  \partial^{\al} \frac{1}{|t-c_{B_Q}|^{d-2}}\right)\\
&=\sum_{\substack {  Q \in \mathcal{D}'_{j,K}\\
 |t-c_{B_Q}|>8 \ell d \sqrt{d} }} \sum _{|\al|\geq 2 }C_{\al, Q} \partial^{\al} \frac{1}{|t-c_{B_Q}|^{d-2}}.\\
\endaligned
\]
The outer sum above is finite, while the inner series  converges uniformly with respect to  $ |t-c_{B_Q}|>8 d \sqrt{d} \ell$. Estimating  $C_{\al, Q}$ by (\ref{ebq0}), while    $\partial^{\al} \frac{1}{|t-c_{B_Q}|^{d-2}}$ is estimated by (\ref{efe}), we have
\[
\aligned
|I_2|
&\leq C \|f\|_{\om}\om(\ell)\sum_{\substack {  Q \in \mathcal{D}'_{j,K}\\
 |t-c_{B_Q}|>8 \ell d \sqrt{d} }} \quad\sum _{|\al|\geq 2 } |\al|^{\frac{1+d}{2}} \left (\frac{4\ell d \sqrt{d}}{|t-c_{B_Q}|}\right )^{d +|\al|-2}.
\endaligned
\]
Again applying $ |t-c_{B_Q}|>8 d \sqrt{d} \ell$, we estimate the inner series as 
\[
\aligned
\sum _{|\al|\geq 2 } |\al|^{\frac{1+d}{2}} \left (\frac{4\ell d \sqrt{d}}{|t-c_{B_Q}|}\right )^{d +|\al|-2}
&=\left (\frac{4\ell d \sqrt{d}}{|t-c_{B_Q}|}\right )^{d}
\sum _{|\al|\geq 2 } |\al|^{\frac{1+d}{2}} \left (\frac{4\ell d \sqrt{d}}{|t-c_{B_Q}|}\right )^{|\al|-2}\\
&\leq  \left (\frac{4\ell d \sqrt{d}}{|t-c_{B_Q}|}\right )^{d }
\sum _{|\al|\geq 2 } |\al|^{\frac{1+d}{2}} 2^{2-|\al|}\\
&\leq C_d  \left (\frac{4\ell d \sqrt{d}}{|t-c_{B_Q}|}\right )^{d }
\endaligned
\]
uniformly with respect to $t$, $ |t-c_{B_Q}|>8 \ell d \sqrt{d} $.

Thus, substituting the above estimate of the inner series into the outer sum for  $I_2$, we obtain
\[
|I_2| \leq C_{d} \|f\|_{\om}\om(\ell)\sum_{\substack {  Q \in \mathcal{D}'_{j,K}\\
 |t-c_{B_Q}|>8 \ell d \sqrt{d}}}  \left (\frac{4\ell d \sqrt{d}}{|t-c_{B_Q}|}\right )^{d }
 \]
with the constant $C_d$.

For each positive integer $k$  define the set
\[
\mathcal{Q}_k=\big\{Q\in \mathcal{D}'_{j,K}:
 |t-c_{B_Q}|\leq  2^{k} \ell d \sqrt{d}  \big\},
\]
and consider the dyadic splitting of the family of cubes
\[
 \{Q\in \mathcal{D}'_{j,K}: \; |t-c_{B_Q}|>8 \ell d \sqrt{d}  \}  = \bigcup_{k\geq 3} \mathcal{Q}_{k+1} \backslash \mathcal{Q}_k .
 \]
Then it holds
\[
\aligned
|I_2|
&\leq  C_{d,K} \|f\|_{\om} \om(\ell)\;\sum_{k\geq 3}\;\sum_{Q\in \mathcal{Q}_{k+1} \backslash \mathcal{Q}_k } \left (\frac{4\ell d \sqrt{d}}{2^{k} \ell\;d \sqrt{d} }\right )^d\\
&\leq  C_{d,K} \;\|f\|_{\om}\;\om(\ell) \;\sum_{k\geq 3}\;\sum_{Q\in \mathcal{Q}_{k+1} \backslash \mathcal{Q}_k } \frac{1}{2^{(k-2)d}}\\
&\leq  C_{d,K}\;\|f\|_{\om}\;\om(\ell)\;\sum_{k\geq 3} \frac{1}{2^{kd}}\;\sum_{Q \in \mathcal{Q}_{k+1} \backslash \mathcal{Q}_k } 1.
\endaligned
\]
Clearly the  inner sum equals to
the number of cubes  $\mathrm{Card} (Q \in \mathcal{Q}_{k+1} \backslash \mathcal{Q}_k ) $, and it is estimated above by the
  number of cubes $\mathrm{Card} (Q \in \mathcal{Q}_{k+1} ) $.

  Also, for each $Q \in \mathcal{Q}_{k+1} \subset \mathcal{D}'_{j,K}$ there exists a   neighbor cube  $Q'\in \mathcal{D}_{j,K}$. For this   cube $Q'$ one has $|t-c_{B_{Q'}}|\leq 2^{k+2} \ell d \sqrt{d} $.
Observing that  each cube $Q' \in \mathcal{D}_{j,K}$ has at most $3^d-1$   neighbor cubes $Q \in \mathcal{D}'_{j,K}$, we can estimate
$\mathrm{Card}\;(Q \in \mathcal{Q}_{k+1})$
by $3^d$ times the number of  cubes
\[
\mathrm{Card}\;\big(Q \in \mathcal{D}_{j,K}: |t-c_{B_{Q'}}|\leq 2^{k+2} \ell d\sqrt{d}\; \big ).
\]
Thus, for each integer $k\geq 3$
 we obtain
\[
\aligned
\sum_{Q\in \mathcal{Q}_{k+1} \backslash \mathcal{Q}_k} 1
& = \mathrm{Card}\;(Q \in \mathcal{Q}_{k+1} \backslash \mathcal{Q}_k)\\
& \leq \mathrm{Card}\;(Q \in \mathcal{Q}_{k+1})\\
&\leq 3^d \mathrm{Card}\; \big(Q \in \mathcal{D}_{j,K}: |t-c_{B_{Q}}|\leq 2^{k+2} \ell d\sqrt{d}\; \big ).
\endaligned
\]
To estimate the last quantity, take the minimal positive integer $m=m(d)$ such that $2^m> d \sqrt{d}$, and
observe that  the ball $B(t, 2^{k+2} \ell d\sqrt{d})$ with the center $t$ and radius $2^{k+2} \ell d\sqrt{d}$  is contained in a cube $\widetilde{Q}$ with the edge size $2^{k+3+m} \ell $. Applying  Lemma \ref{l_cov} to the cube
$\widetilde{Q}$, we  have
\begin{equation}\label{last}
  \sum_{Q\in \mathcal{Q}_{k+1} \backslash \mathcal{Q}_k} 1  \leq 3^d   C_{\lambda}\left (\frac{2^{k+3+m} \ell}{\ell}\right )^{\lambda}  \leq 3^d   C_{\lambda} 2^{(k+3+m)\lambda},
\end{equation}
where the  constants $0< \la <d$ and  $ C_{\lambda} >0$  depend  on the porosity constant $0<c <1$ of $K$.

 Then  (\ref{last})  with the inequality  $0< \lambda < d$ gives
\[
\aligned
|I_2|
&\leq C_{\lambda} 3^d 2^{(3+m)\lambda}\|f\|_{\om}\;\om(\ell) \sum_{k\geq 1} \frac{1}{2^{dk}}\;2^{\lambda k}\\
&\leq  C\|f\|_{\om}\;\om(\ell), \\
\endaligned
\]
 where the constant
 \[C= C_{\lambda} 3^d 2^{(3+m)\lambda} \frac{2^{\lambda }}{2^{d}-2^{\lambda}} \]
  depends on $\la$ and $d$.

With the estimates of $I_1$ and $I_2$ in hand, we are ready to finish the proof of the estimate (\ref{ef11}). For each scale parameter $\ell =2^{-j}$  we obtained   the  function $ \mathcal{F}_{\ell}=F_{\ell}+ O_{\ell}$ which is  harmonic  within  the  specific  neighborhood $K_{\de}$ of the compact $K$, where $\de= \min (2c\ell, \ell)$. This function is the claimed harmonic approximation of the function $f$ in the smaller neighborhood $K_{\de/2}$ of $K$. After the designation,    the proof of Proposition  \ref{pr1} and consequently of Theorem \ref{t2} is completed.

\bibliographystyle{amsplain}

\end{document}